\documentclass{amsart}
\usepackage{amssymb}
\usepackage{amscd}
\usepackage{amsmath}
\usepackage{epsfig}
\usepackage{graphicx}

\usepackage{pgf,tikz}
\usepackage{caption}
\usepackage{shuffle}
 
\usetikzlibrary{arrows}
\usetikzlibrary{decorations.pathreplacing}

\title[AT-associator]
{On the coefficients of \\
the Alekseev Torossian associator}
\author{Hidekazu Furusho}

\address{Graduate School of Mathematics, Nagoya University, 
Chikusa-ku, Furo-cho, Nagoya, 464-8602,  Japan}

\email{furusho@math.nagoya-u.ac.jp}

\date{February 27, 2018.}

\newtheorem{thm}{Theorem}[section]
\newtheorem{lem}[thm]{Lemma}

\theoremstyle{remark}

\theoremstyle{definition}
\newtheorem{defn}[thm]{Definition}

\newtheorem{eg}[thm]{Example}       

\numberwithin{equation}{section}
\numberwithin{figure}{section}

\newcommand{\Conf}{\mathrm{Conf}}
\newcommand{\Aff}{\mathrm{Aff}}

\newcommand{\AT}{\mathrm{AT}}
\newcommand{\KZ}{\mathrm{KZ}}

\newcommand{\RC}{\mathrm{RC}}

\newcommand{\PA}{\mathrm{PA}}

\newcommand{\id}{\mathrm{id}}
\newcommand{\ad}{\mathrm{ad}}

\newcommand{\Aut}{\mathrm{Aut}}
\newcommand{\TAut}{\mathrm{TAut}}
\newcommand{\NTAut}{\mathrm{NTAut}}

\newcommand{\Int}{\mathrm{Int}}

\newcommand{\der}{\mathrm{der}}
\newcommand{\tder}{\mathrm{tder}}
\newcommand{\ntder}{\mathrm{ntder}}

\newcommand{\LieGra}{\mathrm{LieGra}}

\newcommand{\CBH}{\mathrm{CBH}}
\newcommand{\duf}{\mathrm{duf}}

\newcommand{\C}{\mathbb{C}}
\newcommand{\R}{\mathbb{R}}
\newcommand{\Z}{\mathbb{Z}}

\newcommand*\circled[1]{\tikz[baseline=(char.base)]{
            \node[shape=circle,draw,inner sep=2pt] (char) {#1};}}
            
\begin{document}
\bibliographystyle{amsalpha+}
\maketitle

\begin{abstract}
This paper explains a method to calculate 
the coefficients of the Alekseev-Torossian associator
as linear combinations of iterated integrals of Kontsevich weight forms
of Lie graphs.
\end{abstract}

\tableofcontents
\setcounter{section}{-1}
\section{Introduction}\label{introduction}
Associators are group-like non-commutative formal power series with two variables
which were subject to 
the pentagon equation and the hexagon equations in \cite{Dr}
(actually it was shown in \cite{F10} that the former implies the latter).
The notion is involved with wide area of mathematics,
the quantization of Lie-bialgebras (cf. \cite{EK}),
the combinatorial reconstruction of the universal Vassiliev knot invariant 
(cf. \cite{BN,C, KT, LM96a,P}),
the proof of formality of chain operad of  little discs (cf. \cite{Ta, SW}),
the solution of Kashiwara-Vergne conjecture (cf. \cite{AT12}), etc.

A typical example of associators is the {\it KZ-associator} $\varPhi_{\KZ}$ in the algebra $\C\langle\langle A,B\rangle\rangle$
of power series over $\C$ with variables $A$ and $B$, 
which was constructed by two fundamental solutions of the KZ (Knizhnik-Zamolodchikov) equation in \cite{Dr}.
In  \cite{LM} Theorem A.8 and \cite{F} Proposition 3.2.3, it was given a method to calculate
its coefficients 
as linear combinations of
{\it multiple zeta values},
the real numbers defined by the following power series
\begin{equation*}\label{MZV}
\zeta(k_1,\dots,k_m)
:=\sum_{0<n_1<\cdots<n_m}\frac{1}
{n_1^{k_1}\cdots n_m^{k_m}}
\end{equation*}
with $k_1,\dots, k_{m}\in {\mathbb N}$ 
and $k_m>1$ (the condition to be convergent).

The AT-associator $\varPhi_{\AT}$
is another example of associators.
It was introduced by Alekseev and Torossian \cite{AT10}
as an \lq associator in $\TAut_3$'
and later shown to be an associator in $\R\langle\langle A,B\rangle\rangle$
by \v{S}evera and Willwacher \cite{SW}.
It was constructed by a parallel transport of the AT-equation
(cf. \S\ref{sec:Explicit formula of the AT-associator})
on Kontsevich's eye $\overline{C}_{2,0}$
(cf. \S\ref{sec:Kontsevich's eye}).
This paper discusses a AT-counterpart of the results of \cite{LM} and \cite{F}.
We give a method in Theorem \ref{main theorem}
to describe coefficients of 
the AT-associator $\varPhi_{\AT}$
in terms of linear combinations of iterated integrals of Kontsevich weight forms
of Lie graphs (cf. \S \ref{sec:Kontsevich weight forms of Lie graphs})
on $\overline{C}_{2,0}$
and execute computations in lower depth in \S\ref{sec:examples}.

We note that similar (or possibly related) arguments are observed in \cite{Alm} Theorem 8.0.4.5
where Alm described
Lyndon word expansion  of the 1-form $\omega_{\AT}$,
while in this paper we calculate free word expansion of
its parallel transport $\varPhi_{\AT}$.

\section{Kontsevich's eye}\label{sec:Kontsevich's eye}

We will recall the compactified configuration spaces \cite{K03}.
Let $n\geqslant 1$. For a topological space $X$, we define 
$\Conf_n(X):=\{(x_1,\dots,x_n)\bigm| x_i\neq x_j \ (i\neq j)\}$.
The group 
$\Aff_+:=\{x\mapsto ax+b \bigm| a\in\R_+^\times,b\in\C\}$ acts on $\Conf_n(\C)$
diagonally by rescallings and  parallel translations. 
We denote the quotient by 
$$C_n:=\Conf_n(\C)/ \Aff_+$$
for $n\geqslant 2$, which is a connected oriented smooth manifold with dimension $2n-3$.
E.g. $C_2\simeq S^1$ and $C_3\simeq S^1\times (\mathrm{P}^1(\C)\setminus\{0,1,\infty\})$.

Put $\Conf_{n,m}(\mathbb{H},\R):=\Conf_n(\mathbb{H})\times \Conf_m(\R)$
with the coordinate $(z_1,\dots,z_n,x_1,\dots,x_m)$,
where $\mathbb{H}$ is the upper half plane.
The group 
$\Aff^\R_+:=\{x\mapsto ax+b \bigm| a\in\R_+^\times,b\in\R\}$ acts there
diagonally
and we denote the quotient by
$$C_{n,m}:=\Conf_{n,m}(\mathbb{H},\R)/ \Aff^\R_+$$
for $n,m\geqslant 0$ with $2n+m\geqslant 2$,
It is an oriented smooth manifold with dimension $2n+m-2$ and
with $m!$ connected components.
Particularly we denote $C_{n,m}^+$ to be its connected component whose real coordinates
are on the position $x_1<x_2<\cdots<x_m$.
E.g. $C_{0,2}\simeq \{\pm 1\}$,
$C_{1,1}\simeq\{e^{\sqrt{-1}\pi\theta}\bigm| 0<\theta < 1\}$ and 
$C_{2,0}\simeq \mathbb{H}-\{\sqrt{-1}\}$.

Kontsevich constructed 
compactifications $\overline{C}_{n}$ and $\overline{C}_{n,m}^+$ of
$C_n$ and $C_{n,m}^+$ \`{a} la Fulton-MacPherson in \cite{K03}.
They are endowed with the structures of manifolds with corners.
They are functorial  with respect to the inclusions of two finite sets, i.e.
$I_1\subset I_2$ and $J_1\subset J_2$  with 
$\sharp(I_k)=n_k$ and $\sharp(J_k)=m_k$ ($k=1,2$) yield a
natural map $\overline{C}_{n_2,m_2}^+\to\overline{C}_{n_1,m_1}^+$.
The stratification of his compactification has a very nice description in terms of trees in \cite{K03} (also refer \cite{CKTB}).
Particularly $\overline{C}_{2,0}$, which is called {\it Kontsevich's eye},
is given by
$\overline{C}_{2,0}=C_{2,0}\sqcup C_{1,1}\sqcup C_{1,1}\sqcup C_{2}\sqcup C_{0,2}$.
Each component bears a special name
as is indicated in Figure \ref{Eye}.
\begin{figure}[h]
\begin{center}
\begin{tikzpicture}
\draw [very thick] (-3,0).. controls  (-1,2) and (1,2)..(3,0); 
\draw [very thick] (-3,0).. controls  (-1,-2) and (1,-2)..(3,0); 
\draw[dotted] (1.5,1)--(2.3,2.6) node[above]{upper eyelid $C_{1,1}$};
\draw[dotted]  (0.5,-1.5)--(2.5,-2.5) node[below,right]{lower eyelid $C_{1,1}$};
\shadedraw (-3,0).. controls  (-1,2) and (1,2)..(3,0)
(-3,0).. controls  (-1,-2) and (1,-2)..(3,0);
\draw[dotted] (0,0)--(3,-1.5) node[above,right]{iris $C_2$};
\fill[white] (0,0) circle (25pt);
\draw[very thick] (0,0) circle (25pt);
\draw[dotted] (3,0)--(3.75,1.5) node[above,right]{RC (right corner) $C_{0,2}^+$};
\draw[dotted] (-3,0)--(-2,2) node[above]{LC (left corner) $C_{0,2}^+$};
\draw[dotted] (-1.35,-0.35)--(-2.25,-2) node[below]{$C_{2,0}$};
\end{tikzpicture}
\caption{Kontsevich's eye $\overline{C}_{2,0}$}
\label{Eye}
\end{center}
\end{figure}
The {\it upper} (resp. {\it lower}) {\it eyelid} corresponds to 
$z_1$  (resp. $z_2$)  on the the real line.
The {\it iris} magnifies 
collisions of $z_1$ and $z_2$ on $\mathbb{H}$.
LC (resp. RC) which stands for the {\it left} (resp. {\it right}) {\it corner}
is the configuration of $z_1>z_2$  (resp. $z_1<z_2$)
on the real line.

The {\it angle map} $\phi: \overline{C}_{2,0}\to \R/\Z$
is the map induced from the map $\Conf_2(\mathbb{H})\to \R/\Z$
sending 
\begin{equation}\label{eqn:qngle mqp}
\phi:(z_1,z_2)\mapsto \frac{1}{2\pi}\arg\left(\frac{z_2-z_1}{z_2-\bar z_1}\right).
\end{equation}
We note that $\phi$ is identically zero on the upper eyelid.

\section{Kontsevich weight forms of Lie graphs}
\label{sec:Kontsevich weight forms of Lie graphs}
We will recall the notion of  Lie graphs and their Kontsevich weight functions and 1-forms.

\begin{defn}\label{defn:Lie graph}
Let $n\geqslant 1$.  A {\it Lie graph $\Gamma$ of type $(n,2)$} 
is a graph consisting of
two finite sets,
the {\it set of vertices} 
$V(\Gamma):=\{\fbox{1},\fbox{2},\circled{$1$},\circled{$2$},\dots,\circled{$n$}\}$
and the {\it set of edges}
$E(\Gamma)\subset V(\Gamma)\times V(\Gamma)$.
The points $\fbox{1}$ and $\fbox{2}$ 
are called  as the {\it ground points}, while the points
$\circled{$1$},\circled{$2$},\dots,\circled{$n$}$
are called as the {\it air points}.
We equip $V(\Gamma)$ with the total order
$\fbox{1}<\fbox{2}<\circled{$1$}<\circled{$2$}<\dots<\circled{$n$}$.

For each $e\in E(\Gamma)$, under the inclusion 
$E(\Gamma)\subset V(\Gamma)\times V(\Gamma)$,
we call the corresponding first (resp. second) component $s(e)$ (resp. $t(e)$)
as the {\it source} (resp. the {\it target}) of $e$
and denote as $e=(s(e),t(e))$.
We equip $E(\Gamma)$ with the lexicographic order induced from 
that of $V(\Gamma)$.
Both $V(\Gamma)$ and $E(\Gamma)$ 
are subject to the following conditions:
\begin{enumerate}
\item An air point fires two edges:
That means there always exist two edges
with the source $\circled{$i$}$ for each $i=1,\dots,n$.
\item An air point is shot by one edge at most:
That means there exists at most one edge with its target  $\circled{$i$}$ for each $i=1,\dots,n$.
\item A ground point never fire edges:
That means there is no edge with its source on ground points.
\item The graph $\Gamma$ becomes a rooted trivalent tree  after we cut off small neighborhoods of ground points:
That means that the graph of $\Gamma$ admits a unique vertex (called {\it the root}) shoot by no edges 
and it gives a rooted trivalent trees
if we regard the vertex as a root and
distinguish all targets of edges firing ground points.  
\end{enumerate}
\end{defn}

Let $\Gamma$ be a Lie graph of type $(n,2)$. 
Let $ \widehat{\frak f}_2$ be the completed free Lie algebra generated by 
two variables $A$ and $B$.
Lie monomial $\Gamma(A,B)\in \widehat{\frak f}_2$ of degree $n+1$
is defined to be the associated element with the root by the following
procedure:  
With $\fbox{1}$ and $\fbox{2}$, we assign $A$ and $B\in\widehat{\frak f}_2$  respectively.
With each internal vertex $v$ firing two edges $e_1=(v,w_1)$ and $e_2=(v,w_2)$
such that $e_1<e_2$,
we assign $[\Gamma_1,\Gamma_2]\in\widehat{\frak f}_2$
where $\Gamma_1$ and  $\Gamma_2\in\widehat{\frak f}_2$
are the corresponding Lie monomials with the vertices $w_1$ and $w_2$ respectively.
Recursively we may assign Lie elements with all vertices of $\Gamma$.

\begin{figure}[h]
\begin{center}
\begin{tikzpicture}
\node (a) at (0,0) {\fbox{1}};
\node (b) at (2,0) {\fbox{2}};
\node (1) at (3,1.5) {\circled{1}};
\node (2) at (-1,2) {\circled{2}};
\node (3) at  (1,3) {\circled{3}};

\draw[->] (1) to (a);
\draw[->] (1) to (b);
\draw[->] (2) to (1);
\draw[->] (2) to (b);
\draw[->] (3) to (2);
\draw[->] (3) to (b);

\end{tikzpicture}
\caption{$\Gamma(A,B)=[B,[B,[A,B]]]$}
\label{Lie graph example}
\end{center}
\end{figure}
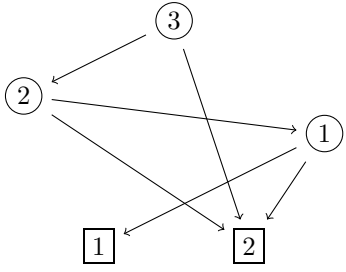
\begin{eg}
Figure \ref{Lie graph example} is an example of Lie graph of type $(3,2)$.
Its root is $\circled{3}$.
The associated Lie elements of the
vertices $\fbox{1},\fbox{2},\circled{$1$},\circled{$2$},\circled{3}$ are 
$A$, $B$, $[A,B]$, $[B,[A,B]]$, $[B,[B,[A,B]]]$ respectively.
\end{eg}
 
Each $e\in E(\Gamma)$ determines a subset $\{s(e),t(e)\}\subset V(\Gamma)$
with $|V(\Gamma)|=n+2$
which yields 
a pull-back 
$\pi:\overline{C}_{n+2,0}\to \overline{C}_{2,0}$.
By composing it with the angle map \eqref{eqn:qngle mqp},
we get a map $\phi_e:\overline{C}_{n+2,0}\to \R/\Z.$
The regular $2n$-forms $\Omega_\Gamma$ on $\overline{C}_{n+2,0}$
(which is  $2n$-dimensional compact space)
associated with $\Gamma$
is given by the ordered exterior product
$$\Omega_\Gamma:=\wedge_{e\in E(\Gamma)} d\phi_e.$$

\begin{defn}
Put $\pi:\overline{C}_{n+2,0}\to \overline{C}_{2,0}$ to be the above projection
induced from the inclusion
$\{\fbox{1},\fbox{2}\}\subset
\{\fbox{1},\fbox{2},\circled{$1$},\circled{$2$},\dots,\circled{$n$}\}$.
The {\it Kontsevich weight function} (see \cite{T}) of $\Gamma$ is the smooth function
$w_\Gamma:\overline{C}_{2,0}\to \C$
defined 
by $w_\Gamma:=
\pi_\ast(\Omega_{\Gamma})$
where $\pi_\ast$ is
the push-forward (the integration along the fiber of the projection $\pi$, cf. \cite{HLTV}), that is, the function
which assigns $\xi\in\overline{C}_{2,0}$ with
$$
w_\Gamma(\xi)
=
\int_{\pi^{-1}(\xi)}\Omega_{\Gamma}\in \C.
$$
\end{defn}

Particularly the special value $w_\Gamma(\RC)$ of the function $w_\Gamma(\xi)$
at $\xi=\RC$ is called the {\it Kontsevich weight} of $\Gamma$. 
It appears as a coefficient of Kontsevich's formula on deformation quantization
in \cite{K03}.
It looks unknown if they are given by MZVs or not.

\begin{defn}
We denote $L\Gamma$ (resp. $R\Gamma$)
to be a graph obtained from $\Gamma$ by adding one more edge $e_L$
from $\fbox{1}$ (resp. $e_R$ from $\fbox{2}$) to the root of $\Gamma$.
The regular $(2n+1)$-form $\Omega_{L\Gamma}$ (resp. $\Omega_{R\Gamma}$) on $\overline{C}_{n+2,0}$
is defined to be
$$
\Omega_{L\Gamma}:=d\phi_{e_L}\wedge \Omega_\Gamma \qquad
(\text{resp.}\quad \Omega_{R\Gamma}:=d\phi_{e_R}\wedge \Omega_\Gamma).
$$
The one-forms
$\omega_{L\Gamma}$ and $\omega_{R\Gamma}$,
which we call the {\it Kontsevich weight forms} of $\Gamma$ here,
are 
the PA
\footnote{
\lq PA' stands for piecewise-algebraic (cf. \cite{KS,HLTV,LV}).}
one-forms of $\overline{C}_{2,0}$
respectively defined by 
$$
\omega_{L\Gamma}:=
\pi_\ast(\Omega_{L\Gamma})
\qquad\text{and} \qquad
\omega_{R\Gamma}:=
\pi_\ast(\Omega_{R\Gamma}),
$$
i.e. they are one-forms respectively defined by
$$
\omega_{L\Gamma}(\xi)
=
\int_{\pi^{-1}(\xi)}\Omega_{L\Gamma}, \quad
\quad \text{and}\quad
\omega_{R\Gamma}(\xi)
=
\int_{\pi^{-1}(\xi)}\Omega_{R\Gamma}
$$
where $\xi$  runs over $\overline{C}_{2,0}$ 
\end{defn}

\section{Main results}
\label{sec:Explicit formula of the AT-associator}
We will recall the definition of the AT-associator and then give a method to calculate its coefficients in Theorem \ref{main theorem}.

Let $\tder_2$ be the Lie algebra consisting of tangential derivations
$\der(\alpha,\beta):\widehat{\frak f}_2\to \widehat{\frak f}_2$
($\alpha,\beta\in\widehat{\frak f}_2$) such that
$A\mapsto [A,\alpha]$ and $B\mapsto [B,\beta]$.
A connection valued there
$$
\omega_\AT=\der\left(\omega_L,\omega_R\right)\in
\tder_2\widehat\otimes\Omega^1_{\PA}(\overline{C}_{2,0})
$$
is introduced in \cite{AT10, T}.
Here $\Omega^1_{\PA}(\overline{C}_{2,0})$ means the space of PA one-forms of
$\overline{C}_{2,0}$
and
\begin{align*}
\omega_L=&\quad d\phi \cdot B+\sum_{n\geqslant 1}\ \sum_{\Gamma\in\LieGra^\mathrm{geom}_{n,2}} \omega_{L\Gamma}\cdot\Gamma(A,B),  \\
\omega_R=&\sigma^*(d\phi) \cdot A+\sum_{n\geqslant 1}\ \sum_{\Gamma\in\LieGra^\mathrm{geom}_{n,2}}  \omega_{R\Gamma}\cdot\Gamma(A,B),
\end{align*}
with the set 
$\LieGra^\mathrm{geom}_{n,2}$ of {\it geometric}
(it means non-labeled) Lie graphs of type $(n,2)$
(cf. Definition \ref{defn:Lie graph}).
We note that both $\Omega_\Gamma$ and $\Gamma(A,B)$ require the order of $E(\Gamma)$ however their product $\Omega_\Gamma\cdot\Gamma(A,B)$
does not (cf. \cite{CKTB}), whence both 
$\omega_L$ and $\omega_R$ do not require labels.
The symbol $\sigma$ stands for the involution of
$\overline{C}_{2,0}$
caused by the switch of $z_1$ and $z_2$.

In \cite{AT10}
they considered the following differential equation on $\overline{C}_{2,0}$
\begin{equation}\label{-AT equation}
dg(\xi)=-g(\xi)\cdot\omega_\AT
\end{equation}
with $\xi\in\overline{C}_{2,0}$,
where the connection $\omega_\AT$
is shown to be flat.
Here $g(\xi)\in \TAut_2:=\exp\tder_{2}$, the pro-algebraic subgroup of $\Aut_2$
consisting of tangential automorphisms
$\Int(\alpha,\beta):\widehat{\frak f}_2\to \widehat{\frak f}_2$
($\alpha,\beta\in\exp\widehat{\frak f}_2$) such that
$A\mapsto \alpha^{-1}A\alpha$ and $B\mapsto \beta^{-1}B\beta$.
They denote its parallel transport (its holonomy) for the straight path from $\alpha_0$ (the position $0$ at the iris, see Figure \ref{parallel}) to $\RC$
by $F_\AT\in\TAut_2$.
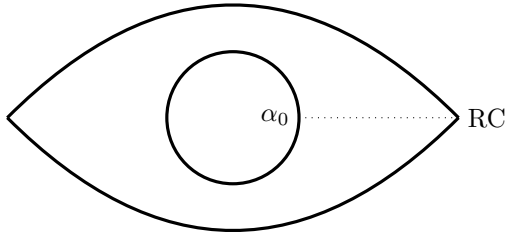
\begin{figure}[h]
\begin{center}
\begin{tikzpicture}
\draw [very thick] (-3,0).. controls  (-1,2) and (1,2)..(3,0); 
\draw [very thick] (-3,0).. controls  (-1,-2) and (1,-2)..(3,0); 
(-3,0).. controls  (-1,-2) and (1,-2)..(3,0);
\draw[very thick] (0,0) circle (25pt);
\draw (3,0) node[right]{RC};
\draw (25pt,0) node[left]{$\alpha_{0}$};
\draw[dotted] (3,0)--(25pt,0);
\end{tikzpicture}
\caption{Parallel transport}
\label{parallel}
\end{center}
\end{figure}

\begin{defn}[\cite{AT10}]
The {\it AT-associator} $\varPhi_{\AT}$ is defined to be 
\begin{equation}\label{defn:AT-associator}
\varPhi_{\AT}:=F^{1,23}_\AT\circ F^{2,3}_\AT\circ (F^{1,2}_\AT)^{-1}\circ (F^{12,3}_\AT)^{-1}
\in\TAut_3.
\end{equation}
Here for any $T=\Int(\alpha,\beta)\in \TAut_2$, we denote
$$T^{1,2}:=\Int\left(\alpha(A,B),\beta(A,B),1\right),\quad
T^{2,3}:=\Int\left(1,\alpha(B,C),\beta(B,C)\right),$$
$$T^{1,23}:=\Int\left(\alpha(A,B+C),\beta(A,B+C),\beta(A,B+C)\right),$$
$$T^{12,3}:=\Int\left(\alpha(A+B,C),\alpha(A+B,C),\beta(A+B,C)\right)$$
in $\TAut_3:=\exp\tder_3$ which is similarly defined to be the group of
tangential automorphisms of the completed free Lie algebra $\widehat{\frak f_3}$
with variables $A$, $B$ and $C$.
\end{defn}

We note that there is a Lie algebra inclusion 
$\widehat{\frak f}_2\hookrightarrow\tder_3$ sending
\begin{equation}
A\mapsto t_{12}:=\der (B,A,0)  \quad \text{and} \quad
B\mapsto t_{23}:=\der (0,C,B)
\end{equation}
which induces an inclusion $\exp\widehat{\frak f}_2\hookrightarrow \TAut_3$.

\begin{thm}[\cite{AT12,SW}]\label{thm:ATSW}
The AT-assocciator $\varPhi_{\AT}$ forms an associator. Namely it belongs
to $\exp\widehat{\frak f}_2$  
($\subset\C\langle\langle A,B\rangle\rangle$) and satisfies the equations 
\cite{Dr} (2.12), (2.13) and (5.3).
Furthermore  it is real (i.e. it
belongs to the real structure $\R\langle\langle A,B\rangle\rangle$)
and even  (i.e. $\varPhi_{\AT}(-A,-B)=\varPhi_{\AT}(A,B)$.)
\end{thm}

As we saw in the introduction,
all the coefficients of the KZ-associator  $\varPhi_{\KZ}$ are 
known to be explicitly described in terms of multiple zeta values,
which are iterated integrals of two differential 1-forms of $dt/t$ and $dt/(1-t)$ on $\mathbb P^1\setminus\{0,1,\infty\}$.
The following theorem would be its AT-counterpart, that is,
we give a method to present all the coefficients of 
the AT-associator  $\varPhi_{\AT}$
in terms of linear combinations of iterated integrals of Kontsevich weight forms, on $\overline{C}_{2,0}$, of Lie graphs .

\begin{thm}\label{main theorem}
We have
\begin{equation}\label{main formula}
\varPhi_{\AT}=\left(\mathcal{P}\exp\int_{\RC}^{\alpha_0}
(l_{\widehat\omega}+D_{\widehat\omega})\right)
(1) \
\in\C\langle\langle A,B\rangle\rangle.
\end{equation}
Here
$l_{\widehat\omega}$ is the left multiplication by ${\widehat\omega}$ and
$D_{\widehat\omega}$ is given by
$$
D_{\widehat\omega}
:=\der\left(0,\widehat{\omega}\right)\in
\tder_2\widehat\otimes\Omega^1_{\PA}(\overline{C}_{2,0})
$$
with
\begin{equation}
\widehat{\omega}
:=\sum_{n\geqslant 1} \ \sum_{\Gamma\in\LieGra_{n,2}}
\widehat\omega_{\Gamma}\cdot\Gamma(A,B)
\quad \text{ and }\quad
\widehat\omega_{\Gamma}:=\omega_{R\Gamma}-\omega_{L\Gamma}.
\end{equation}
and for any one-form $\Omega\in\Omega^1_{\PA}(\overline{C}_{2,0})$ we define 
\begin{align*}
\mathcal{P}&\exp\int_{\RC}^{\alpha_0}\Omega
:=\id+\int_{\RC}^{\alpha_0}\Omega+
\int_{\RC}^{\alpha_0}\Omega\cdot\Omega +
\cdots \\
&:=
\id+\int_{0<s_1<1}\ell^*\Omega(s_1)+
\int_{0<s_1<s_2<1}\ell^*\Omega(s_2)\wedge\ell^*\Omega(s_1) +\cdots.
\end{align*}
\end{thm}

The right hand side of \eqref{main formula} 
is directly computable
as will be given in \S\ref{sec:examples} \eqref{direct formula}. 

\begin{proof}
Since $\varPhi_{\AT}$ is an associator,
we have
\begin{equation}\label{PhiAT=muAT}
\varPhi_{\AT}(t_{12},t_{23})=\mu^{1,23}_{\varPhi_{\AT}}\circ \mu^{2,3}_{\varPhi_{\AT}}\circ 
\{\mu^{1,2}_{\varPhi_{\AT}}\}^{-1}\circ \{\mu^{12,3}_{\varPhi_{\AT}}\}^{-1}
\end{equation}
by \cite{AET} Theorem 2.
Here $\mu_{\varPhi_{\AT}}$ is the automorphism of
$\C\langle\langle A,B\rangle\rangle$
such that
\begin{align*}
A&\mapsto \qquad\qquad \varPhi_{\AT}(A,C)\cdot A\cdot \varPhi_{\AT}(A,C)^{-1},\\
B&\mapsto \exp(C/2)\varPhi_{\AT}(B,C)\cdot B\cdot\varPhi_{\AT}(B,C)^{-1}\exp(-C/2)
\end{align*}
with $C=-A-B$.
We note that
$\mu_{\varPhi_{\AT}}(\CBH(A,B))=A+B$.

Let $\mathrm{Sol KV}(\C)$ be the set  of solutions  of the generalized Kashiwara-Vergne problem,  that is, the set of 
$p\in \TAut_2$
satisfying
$p(\CBH(A,B))=A+B$
and the coboundary Jacobian condition 
$
\delta\circ J(p)=0.
$
Here
$J$ stands for the Jacobian cocycle $J:\TAut_2\to\frak{tr}_2$
and $\delta$ denotes the differential map $\delta:\frak{tr}_n\to\frak{tr}_{n+1}$
for $n=1,2,\dots$
with the trace space $\frak{tr}_n$
(for their precise definitions see \cite{AT10, AET}).

Recall that the element $F_\AT\in\TAut_2$ was defined to be the parallel transport of the differential equation
\eqref{-AT equation} from $\alpha_0$ to $\RC$.
By the following lemma, it is also regarded as  a parallel transport
of the differential equation \eqref{AT-equation},
which we call  {\it the AT-equation}.

\begin{lem}\label{FAT=ptAT}
The element $F_\AT$ is also defined to be the parallel transport, for
the straight path $\ell(s)$ ($0\leqslant s \leqslant 1$)
from RC to $\alpha_0$ in Figure \ref{parallel},
of the differential equation
(
{\it the AT-equation})
on $\overline{C}_{2,0}$
\begin{equation}\label{AT-equation}
dG(\xi)=\omega_\AT\cdot G(\xi)
\end{equation}
with $G(\xi)\in \TAut_2$.
\end{lem}

\begin{proof}
Let $g_0(\xi)$ be the unique local solution of \eqref{-AT equation}
with the initial condition $g_0(\alpha_0)=1$ and
$G_\RC(\xi)$ be the unique local solution of \eqref{-AT equation}
with the initial condition $G_\RC(\RC)=1$.
We can check that $g_0(\xi)^{-1}$ is a solution of the AT-equation 
\eqref{AT-equation} by direct computations.
Then by the uniqueness of the solution of the differential equation, 
$G_\RC(\xi)$ and $g_0(\xi)^{-1}$ are different by the right multiplication.
Since  $G_\RC(\RC)=1$ and $g_0(\RC)^{-1}=F_\AT^{-1}$,
we must have 
$
G_\RC(\xi)=g_0(\xi)^{-1}F_\AT.
$
So the parallel transport which we want is
$G_\RC(\alpha_0)=g_0(\alpha_0)^{-1}F_\AT=F_\AT$.
\end{proof}
By Lemma \ref{FAT=ptAT}, 
we have 
\begin{equation}\label{FAT=Pexp}
F_\AT=\mathcal{P}\exp\int_{\RC}^{\alpha_0}\omega_{\AT} \in \TAut_2.
\end{equation}


\begin{lem}\label{FAT in SolKV}
$F_\AT\in\mathrm{Sol KV}(\C).$
\end{lem}				

\begin{proof}
By the differential equation \cite{AT10} (8) for 
$\mathrm{ch}_\xi\in\widehat{\frak f}_2$ ($\xi\in\overline{C}_{2,0}$)
$$
d\mathrm{ch}_\xi=\omega_\AT\cdot \mathrm{ch}_\xi ,
$$
$g_0(\xi)(\mathrm{ch}_\xi)$ is constant where $g_0(\xi)$ is the solution of \eqref{-AT equation} introduced in the proof of Lemma \ref{FAT=ptAT}.
We have $F_{\AT}(\CBH(A,B))=A+B$
because $g_0(\alpha_0)=\id$, $g_0(\RC)=F_\AT$
and
$\mathrm{ch}_{\alpha_0}=A+B$, $\mathrm{ch}_{\RC}=\CBH(A,B)$.

While we have a differential equation on $\overline{C}_{2,0}$
$$
dJ(g_0(\xi))=\omega_\AT\cdot J(g_0(\xi))+\mathrm{div}(\omega_\AT)
$$
here $\mathrm{div}:\tder_2\to \frak{tr}_2$ is the divergence cocycle
which integrates to the Jacobian cocycle $J$  (cf. \cite{AT10,AT12}).
It is verified by considering a tangent vector $v$ at the point $\xi\in \overline{C}_{2,0}$ 
and evaluating  the both hands sides at $v$. 
In more detail, 
\begin{align*}
v(J(g_0(\xi))) &= d/dt \{J\left(g_0(\xi+tv)\right) - J\left(g_0(\xi)\right)\} |_{t=0} \\
            &= d/dt \{ J\left(g_0(\xi+tv) g_0(\xi)^{-1} g_0(\xi)\right) - J\left(g_0(\xi)\right)\} |_{t=0}\\
            &= d/dt \{ J\left(g_0(\xi+tv) g_0(\xi)^{-1}\right) + \left(g_0(\xi+tv) g_0(\xi)^{-1}\right) J\left(g_0(\xi)\right)\}|_{t=0}\\
            &= \mathrm{div}\left(\omega_\AT(v)\right) + \omega_\AT(v) J\left(g_0(\xi)\right).
\end{align*}
By combining the above differential equation with the differential equation \cite{AT10} (9) for $\duf_\xi\in\frak{tr}_2$
$$
d\duf_\xi(A,B)=\omega_\AT\cdot \duf_\xi(A,B)+\mathrm{div}(\omega_\AT) ,
$$
we obtain that
$
g_0(\xi)\cdot \bigl\{\duf_\xi -J(g_0(\xi))\bigr\}
$
is constant.
By 
$\duf_{\alpha_0}=0$, 
we have $\duf_\xi =J(g_0(\xi))$.
Since $\duf_{\RC}$ is the Duflo function which lies on $\ker \delta$,
we have
$\delta\circ J(F_\AT)=0$ by  $g_0(\RC)=F_\AT$.
\end{proof}

We have $F_\AT\in\mathrm{Sol KV}(\C)$ by Lemma \ref{FAT in SolKV} and
$\mu_\varphi\in\mathrm{Sol KV}(\C)$ for all associators $\varphi\in\exp\widehat{\frak f}_2$
by \cite{AT12, AET}.
Both $F_\AT$ and $\mu_{\varPhi_{\AT}}$ give rise to the same $\varPhi=\varPhi_{\AT}$ by
\eqref{defn:AT-associator} and \eqref{PhiAT=muAT}.
So, by \cite{AT12} Proposition 7.2
\footnote{
We follow the convention in \cite{AET}, where products are expressed in a way opposite to those of \cite{AT12}.},
we have $F_\AT=\exp (\lambda t)\circ \mu_{\varPhi_{\AT}}$ for some $\lambda\in\C$
and $t=\der(B,A)=\ad(-A-B)$.
Since $t$ is an inner derivation, we have
\begin{equation}\label{FAT=muAT}
F_\AT \equiv \mu_{\varPhi_{\AT}}  \pmod{\Int_2}
\end{equation}
where $\Int_2$ is the (normal) subgroup of $\TAut_2$
consisting of inner automorphisms.

By the 2- and 3-cycle relation (\cite{Dr} (2.12) and (5.3)) for  $\varPhi_\AT$,
\begin{equation}\label{muAT=mu0AT}
\mu_{\varPhi_{\AT}} \equiv \mu_{\varPhi_{\AT},0} \pmod{\Int_2}
\end{equation}
where $\mu_{\varPhi_{\AT},0}$
is the automorphism such that 
\begin{align}\label{NTAut_2}
A\quad &\mapsto \qquad\qquad\qquad A  ,\\
B\quad &\mapsto \varPhi_{\AT}(A,B)^{-1}\cdot B \cdot\varPhi_{\AT}(A,B). \notag
\end{align}

\begin{lem}\label{lem:muAT0=pAT}
The automorphism $\mu_{\varPhi_{\AT},0}$
is also given by the parallel transport of the differential equation
\begin{equation}\label{normalized AT-equation}
dG=D_{\widehat\omega}\cdot G
\end{equation}
from $\RC$ to $\alpha_0$, i.e.
\begin{equation}\label{mu0=Pexp}
\mu_{\varPhi_{\AT},0}=\mathcal{P}\exp\int_{\RC}^{\alpha_0}D_{\widehat\omega}. 
\end{equation}
\end{lem}

\begin{proof}
The derivation $D_{\widehat\omega}$ is a (unique) one with $\deg\geqslant 2$
which is congruent to $\omega_\AT$ modulo inner derivations,
whence we have
\begin{equation}\label{PexpAT=PexpD mod Int}
\mathcal{P}\exp\int_{\RC}^{\alpha_0}\omega_{\AT}
\equiv
\mathcal{P}\exp\int_{\RC}^{\alpha_0}D_{\widehat\omega} \mod\Int_2.
\end{equation}
By \eqref{FAT=Pexp}, \eqref{FAT=muAT}, \eqref{muAT=mu0AT}
and \eqref{PexpAT=PexpD mod Int},
we obtain \eqref{mu0=Pexp} modulo $\Int_2$.
Since the right hand side of \eqref{PexpAT=PexpD mod Int} is the automorphism which yields the identity modulo $\deg\geqslant 2$,
and
stabilizes the element $A$ and
the conjugacy class of $B$ as in \eqref{NTAut_2},
and $\mu_{\varPhi_{\AT},0}$ is the unique automorphism in the same residue class
 with the same property  and $\varPhi_{\AT}\equiv 1 \mod \deg\geqslant 2$,
we obtain \eqref{mu0=Pexp}.
Thus Lemma \ref{lem:muAT0=pAT} is proven.
\end{proof}

Let us go back to the proof of Theorem \ref{main theorem}.
As is explained in \cite{DG}, there is a natural Lie algebra isomorphism
\begin{equation}\label{identification:Lie level}
\ntder\Pi_0 \simeq\ntder \Pi_{10}
\end{equation}
whose correspondence is given by 
$\der(0,\phi)\mapsto l_\phi+\der(0,\phi)$ with $\phi\in\widehat{\frak f}_2$ where 
$l_\phi$ is the left multiplication by $\phi$.
(Here $\Pi_0$ and $\Pi_{10}$ are notations employed in \cite{DG}.
Both of them stand for  $\C\langle\langle A,B\rangle\rangle$ but the latter is regarded as a \lq right-torsor' over the former.)
The Lie algebra isomorphism \eqref{identification:Lie level} is integrated to 
the group isomorphism
\begin{equation}\label{identification:group level}
\NTAut\Pi_0 \simeq\NTAut \Pi_{10},
\end{equation}
where the former is the subgroup of $\TAut_2$
consisting of the automorphism $\mu_{\Phi,0}$ 
given as \eqref{NTAut_2} for
$\Phi\in\exp\widehat{\frak f}_2$
while the latter is the group of automorphisms of 
$\C\langle\langle A,B\rangle\rangle$
consisting of $l_\Phi\circ\mu_{\Phi,0}$
and the above identification is given by
$\mu_{\Phi,0} \mapsto l_\Phi\circ\mu_{\Phi,0}$. 

We have $D_{\widehat\omega}\in \ntder\Pi_0\otimes\Omega^1_{\PA}(\overline{C}_{2,0})
 $.
By Lemma \ref{lem:muAT0=pAT},
$\mu_{\varPhi_{\AT},0}$
is the parallel transport 
of the differential equation \eqref{normalized AT-equation},
whence it
lies on $\NTAut\Pi_0$.
And under the identification 
\eqref{identification:group level}
it corresponds to
$l_{\varPhi_{\AT}}\circ\mu_{\varPhi_{\AT},0}$,
which is also
given as the parallel transport of the differential equation
\begin{equation}\label{normalized AT-equation: 01version}
dG=\{l_{\widehat\omega}+D_{\widehat\omega}\}\cdot G
\end{equation}
from $\RC$ to $\alpha_0$.
So
\begin{equation}
l_{\varPhi_{\AT}}\circ\mu_{\varPhi_{\AT},0}
=
\mathcal{P}\exp\int_{\RC}^{\alpha_0}
(l_{\widehat\omega}+D_{\widehat\omega})
\ \in\NTAut \Pi_{10}.
\end{equation}
By applying to $1_{10}:=1\in \Pi_{10}=\C\langle\langle A,B\rangle\rangle$, we obtain \eqref{main formula},
whence Theorem \ref{main theorem} is proven.
\end{proof}

\section{Coefficients in depth 1 and 2}\label{sec:examples}
Along the method of Theorem \ref{main theorem},
we will execute calculations of  the coefficients of
the AT-associator in lower depth terms in the equations
\eqref{eq:depth=1} and \eqref{eq:depth=2}.

By definition, we have
\begin{align*}
\mathcal{P}\exp\int_{\RC}^{\alpha_0}&(l_{\widehat\omega}+D_{\widehat\omega})
=\id + \int_\RC^{\alpha_0}(l_{\widehat\omega}+D_{\widehat\omega}) 
+\int_\RC^{\alpha_0}(l_{\widehat\omega}+D_{\widehat\omega})
\cdot(l_{\widehat\omega}+D_{\widehat\omega}) \\
&+\int_\RC^{\alpha_0}(l_{\widehat\omega}+D_{\widehat\omega})
\cdot(l_{\widehat\omega}+D_{\widehat\omega})
\cdot(l_{\widehat\omega}+D_{\widehat\omega})
+\cdots \in\NTAut \Pi_{10}.
\end{align*}
Then from \eqref{main formula} it follows
\begin{align}\label{direct formula}
&\varPhi_{\AT}=1+\int_\RC^{\alpha_0}{\widehat\omega}+
\left\{\int_\RC^{\alpha_0}{\widehat\omega}\cdot{\widehat\omega}+
\int_\RC^{\alpha_0}D_{\widehat\omega}({\widehat\omega})\right\} \\
&+\left\{
\int_\RC^{\alpha_0}{\widehat\omega}\cdot{\widehat\omega}\cdot{\widehat\omega}
+2\int_\RC^{\alpha_0}{\widehat\omega}\cdot D_{\widehat\omega}({\widehat\omega})
+\int_\RC^{\alpha_0} D_{\widehat\omega}({\widehat\omega})\cdot{\widehat\omega}
+\int_\RC^{\alpha_0} D_{\widehat\omega}\circ D_{\widehat\omega}({\widehat\omega})
\right\} \notag \\
&+\cdots\cdots
\in\exp\widehat{\frak f}_2\subset\C\langle\langle A,B\rangle\rangle. \notag
\end{align}

\medskip

{\it Depth $1$ case} :
Recall that $(\varPhi_{\KZ}|A^{n-1}B)$,
the coefficient of $A^{n-1}B$ ($n\geqslant 2$) 
in the KZ-associator $\varPhi_{\KZ}$ (\cite{Dr})
is calculated to be
$$(\varPhi_{\KZ}|A^{n-1}B)= -\zeta(n)$$
with the Riemann zeta value $\zeta(n)=\sum_{0<k}\frac{1}{k^n}$ in \cite{Dr}.
By \eqref{direct formula}, 
its analogue in the AT-associator $\varPhi_{\AT}$ is calculated as follows:

\begin{eg}\label{eg:depth=1}
For $n\geqslant 2$, we have
\begin{equation}\label{eq:depth=1}
(\varPhi_{\AT}|A^{n-1}B)=\left(\int_\RC^{\alpha_0}{\widehat\omega}\Bigm|A^{n-1}B\right)=\int_\RC^{\alpha_0}{\widehat\omega}_{\Gamma_{n-1}}.
\end{equation}
Here $\Gamma_{n-1}$ is the Lie graph of type $(n-1,2)$ with $\Gamma_{n-1}(A,B)=(\ad A)^{n-1}(B)$,
whose geometric graph is depicted in Figure \ref{Gamma single index}.

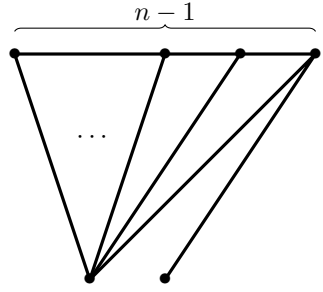
\begin{figure}[h]
\begin{center}
         \begin{tikzpicture}
                  \draw[-,very thick] (-1,3) node{$\bullet$} --(1,3) node{$\bullet$}--(2,3) node{$\bullet$}--(3,3) node{$\bullet$};
                  \draw[-,very thick] (0,0) node{$\bullet$}--(-1,3);
                  \draw[-,very thick] (0,0)--(1,3);
                  \draw[-,very thick] (0,0)--(2,3);
                  \draw[-,very thick] (0,0)--(3,3)--(1, 0) node{$\bullet$}; 
                  \draw[decorate,decoration={brace}] (-1,3.3) -- (3,3.3) node[midway,above]{$n-1$};
                  \draw (-0.3,1.9) node[right]{$\cdots$};
         \end{tikzpicture}
\caption{$\Gamma_{n-1}$}
\label{Gamma single index}
\end{center}
\end{figure}
Actually by Theorem \ref{thm:ATSW}, it can be deduced that
$$(\varPhi_{\AT}|A^{n-1}B)=\frac{B_n}{2(n!)}$$
for $n\geqslant 2$
with the Bernoulli number $B_n$ (see also \cite{Alm}).
\end{eg}

{\it Depth $2$ case} :
We also remind that the coefficient of $A^{b-1}BA^{a-1}B$
($a\geqslant 1, b>1$) 
in the KZ-associator $\varPhi_{\KZ}$
is calculated to be
$$(\varPhi_{\KZ}|A^{b-1}BA^{a-1}B)=\zeta(a,b)$$
with the double zeta value $\zeta(a,b)=\sum_{0<k<l}\frac{1}{k^al^b}$
in \cite{LM} (cf. \cite{F}).
Again by \eqref{direct formula}, 
its analogue in $\varPhi_{\AT}$ is calculated as follows:

\begin{eg}\label{eg:depth=2}
For $a,b\geqslant 1$, we have
\begin{align*}
(\varPhi_{\AT}|A^{b-1}BA^{a-1}B)
=\left(\int_\RC^{\alpha_0}{\widehat\omega}
+\int_\RC^{\alpha_0}{\widehat\omega}\cdot{\widehat\omega}
+\int_\RC^{\alpha_0}D_{\widehat\omega}({\widehat\omega})\Bigm|A^{b-1}BA^{a-1}B\right).
\end{align*}

Let $\Gamma_{i,j,k}$ ($i\geqslant 1$, $j,k\geqslant 0$) be the Lie graph 
with
$\Gamma_{i,j,k}(A,B)\\=(\ad A)^{i-1}[(\ad A)^j(B), (\ad A)^k(B)]
=\sum\limits_{i_1+i_2=i-1\atop 0\leqslant i_1,i_2}\binom{i-1}{i_1}
[(\ad A)^{j+i_1}(B),(\ad A)^{k+i_2}(B)]$,
whose geometric graph is 
depicted in Figure \ref{Gamma triple index}.

\begin{figure}[h]
\begin{center}
         \begin{tikzpicture}
                  \draw[-, very thick] (0,0) node{$\bullet$} --(3.8,1.4) node{$\bullet$} --(3.25,2.5) node{$\bullet$} --(3.1,2.8) node{$\bullet$}-- (0,0);
                  \draw[decorate,decoration={brace,mirror}] (4,1.4) -- (3.2,3) node[midway,right]{$k$};
                  \draw (3.2,1.8) node[rotate=120]{$\cdots$};              
                   \draw[-,very thick] (1,0) --(3.8,1.4);
                  \draw[-,very thick] (0,0) node{$\bullet$} --(3.25,2.5);
                  \draw[-,very thick] (1,0) node{$\bullet$} --(1.5,2)node{$\bullet$}--(1.2,2)node{$\bullet$}--(0.6,2)node{$\bullet$}--(0.3,2)node{$\bullet$}--(0,0);
                   \draw[-,very thick] (0,0) --(1.5,2);
                  \draw[-,very thick] (0,0) --(1.2,2);
                  \draw[-,very thick] (0,0) --(0.6,2);
                  \draw (0.5,1.75) node[right]{\tiny{$\cdots$}};               
                  \draw[decorate,decoration={brace}] (0.3,2.2) -- (1.5,2.2) node[midway,above]{$j$};
                  \draw[-,very thick] (3.1,2.8) --(-1,3.5)node{$\bullet$}--(0.3,2) ;
                  \draw[-,very thick] (-1,3.5)--(-1.2,3)node{$\bullet$} --(-2,1)node{$\bullet$}--(-2.2,0.5)node{$\bullet$} --(0,0) ;
                  \draw[-,very thick] (-1.2,3)--(0,0);    
                  \draw[-,very thick] (-2,1)--(0,0);
                  \draw[decorate,decoration={brace,mirror}] (-1.1,3.7) -- (-2.4,0.5) node[midway,left]{$i$};
                  \draw (-1.4,1.6) node[rotate=70]{$\cdots$};              
          \end{tikzpicture}
\caption{$\Gamma_{i,j,k}$}
\label{Gamma triple index}
\end{center}
\end{figure}
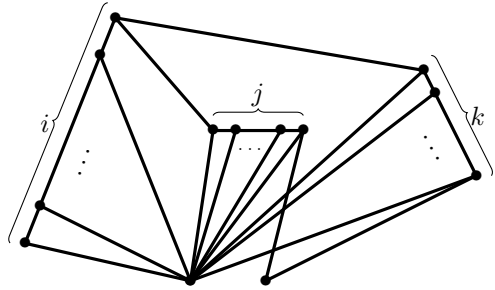

Each term  is calculated as follows:
\begin{align*}
&\left(\int_\RC^{\alpha_0}{\widehat\omega} 
\Bigm|A^{b-1}BA^{a-1}B\right) \\
&=\sum_{i+j+k=a+b-1\atop j<k,\ 1\leqslant i, \ 0\leqslant j,k}\sum_{i_1+i_2=i-1\atop 0\leqslant i_1,i_2}
\binom{i-1}{i_1}
 \left\{
    (-1)^{j+i_1-b+1}\binom{j+i_1}{b-1}+(-1)^{k+i_2-b}\binom{k+i_2}{b-1}
  \right\} \\
&\qquad\qquad \cdot \int_\RC^{\alpha_0}{\widehat\omega}_{\Gamma_{i,j,k}}. \\
&\left(\int_\RC^{\alpha_0}{\widehat\omega} \cdot {\widehat\omega}
\Bigm|A^{b-1}BA^{a-1}B\right) 
=\sum_{i+j=a+b\atop 2\leqslant i,j}(-1)^{i-b}\binom{i-1}{b-1}
\int_\RC^{\alpha_0}
{\widehat\omega}_{\Gamma_{i-1}}\cdot {\widehat\omega}_{\Gamma_{j-1}}. \\
&\left(\int_\RC^{\alpha_0}D_{\widehat\omega}({\widehat\omega})\Bigm|A^{b-1}BA^{a-1}B\right) \\
&=\sum_{i+j=a+b \atop 2\leqslant i,j}\sum_{j_1+j_2=j-1 \atop 0\leqslant j_1,j_2}
\binom{j-1}{j_1}
 \left\{
    (-1)^{j_1-b+1}\binom{j_1}{b-1}-(-1)^{j_2+i-b}\binom{j_2+i-1}{b-1}
  \right\} \\
&\qquad\qquad \cdot \int_\RC^{\alpha_0}
{\widehat\omega}_{\Gamma_{i-1}}\cdot {\widehat\omega}_{\Gamma_{j-1}}.
\end{align*}
Therefore the coefficient is given by a rational linear combination of 
single integrals and double iterated integrals of Kontsevich weight forms of Lie graphs as follows:
\begin{equation}\label{eq:depth=2}
(\varPhi_{\AT}|A^{b-1}BA^{a-1}B)
=\sum_{i+j+k=a+b-1\atop j<k,\ 1\leqslant i, \ 0\leqslant j,k}c_{i,j,k}
\int_\RC^{\alpha_0}{\widehat\omega}_{\Gamma_{i,j,k}}+
\sum_{i+j=a+b\atop 2\leqslant i,j}c_{i,j}\int_\RC^{\alpha_0}
{\widehat\omega}_{\Gamma_{i-1}}\cdot {\widehat\omega}_{\Gamma_{j-1}}
\end{equation}
with
\begin{align*}
&c_{i,j,k}
=
\sum_{i_1+i_2=i-1\atop 0\leqslant i_1,i_2}
\binom{i-1}{i_1}
  \left\{
    (-1)^{j+i_1-b+1}\binom{j+i_1}{b-1}+(-1)^{k+i_2-b}\binom{k+i_2}{b-1}
  \right\}, \\
&c_{i,j}=
 (-1)^{i-b}
  \left\{\binom{i-1}{b-1}+\delta_{j,b}\binom{j-1}{b-1}\right\}
-\sum_{j_1+j_2=j-1 \atop 0\leqslant j_1,j_2}(-1)^{j_2}
\binom{j-1}{j_1}
    \binom{j_2+i-1}{b-1}.
\end{align*}
Here $\delta_{i,j}$ is the Kronecker delta symbol.
We note that
$
(\varPhi_{\AT}|A^{b-1}BA^{a-1}B)
=0
$
when $a+b$ is odd,
because $\varPhi_{\AT}$ is even.
\end{eg}

In principle, our formula \eqref{main formula} enables us to calculate explicitly all the coefficients of the AT-associator $\varPhi_{\AT}$
as rational linear combinations of iterated integrals of
Kontsevich weight forms of Lie graphs as \eqref{eq:depth=1} and \eqref{eq:depth=2}.

There are explicit formulae to describe
all the coefficients of the KZ-associator $\varPhi_{\KZ}$
in terms of multiple zeta values (cf. \cite{F,LM}).
Whereas 
any explicit formulae
to present all the coefficients of $\varPhi_{\AT}$
(or, more generally, iterated integrals of Kontsevich weight forms of Lie graphs)
as linear combinations of multiple zeta values
looks unknown
\footnote{
The author learned that Felder \cite{Fe} calculted  
$(\varPhi_{\AT}|A^{2}BA^{4}B)=\frac{2048\zeta(3,5)-6293\zeta(3)\zeta(5)}{524288\pi^8}.$
}
though it was claimed 
in \cite{RW} Remark 1.5 that all coefficients of $\varPhi_{\AT}$ are linear combinations
of multiple zeta values with rational coefficients.
If their claim is turned into an effective algorithm computing the coefficients of $\varPhi_{\AT}$,
it would give formulae for coefficients of $\varPhi_{\AT}$
in terms of multiple zeta values rather than iterated integrals of  Kontsevich weight forms.

\bigskip
\thanks{
{\it Acknowledgements}.
This work was supported by JSPS KAKENHI 15KK0159
and Daiko Foundation.
The author thanks for Anton Alekseev for explaining him the  proof of
Lemma \ref{FAT in SolKV}
and also thanks for Johan Alm and Matteo Felder who pointed him their works.
He is also grateful to the referee
whose comments improved this paper.
}



\begin{thebibliography}{Utah}

\bibitem[AET]{AET}
Alekseev, A., Enriquez, B. and Torossian, C.,
{\it Drinfeld associators, braid groups and explicit solutions of the Kashiwara-Vergne equations},
Publ. Math. Inst. Hautes \'{E}tudes Sci. No. {\bf 112} (2010), 143--189. 

\bibitem[AT1]{AT10}
Alekseev, A. and Torossian, C.,
{\it Kontsevich deformation quantization and flat connections},
Comm. Math. Phys. {\bf 300} (2010), no. 1, 47--64.

\bibitem[AT2]{AT12}
Alekseev, A. and Torossian, C.,
{\it The Kashiwara-Vergne conjecture and Drinfeld's associators},
Ann. of Math. (2) {\bf 175} (2012), no. 2, 415--463. 

\bibitem[Alm]{Alm}
Alm, J.,
{\it Universal algebraic structures on polyvector fields},
PhD thesis, Stockholm Univ, 2014.

\bibitem[B]{BN}
Bar-Natan, D.;
{\it Non-associative tangles},
Geometric topology (Athens, GA, 1993), 139-183, 
AMAPIP Stud. Adv. Math., {\bf 2.1}, Amer. Math. Soc., Providence, RI, 1997.


\bibitem[Ca]{C}
Cartier, P.;
{\it Construction combinatoire des invariants de Vassiliev-Kontsevich des n\oe uds},
C. R. Acad. Sci. Paris Ser. I Math. {\bf 316} (1993), no. 11, 1205-1210. 
 
\bibitem[CKTB]{CKTB}
Cattaneo, A., Bernhard, B., Torossian, C. and Brugui\`{e}res, A.,
{\it D\'{e}formation, quantification, th\'{e}orie de Lie},
Panoramas et Synth\`{e}ses, {\bf 20}. Soci\'{e}t\'{e} Math\'{e}matique de France, Paris, 2005. 
 
\bibitem[DG]{DG} 
Deligne, P. and  Goncharov, A., 
 {\it Groupes fondamentaux motiviques de Tate mixte}, 
 Ann. Sci. \'{E}cole Norm. Sup. (4) {\bf 38} (2005), no. 1, 1--56.

\bibitem[Dr]{Dr}
Drinfeld, V. G.,
{\it On quasitriangular quasi-Hopf algebras and
a group closely connected with ${\rm Gal}(\overline{\mathbb Q}/{\mathbb Q})$},
Leningrad Math. J. {\bf 2} (1991), no. 4, 829--860.

\bibitem[EK]{EK}
Etingof, P. and Kazhdan, D.;
{\it Quantization of Lie bialgebras. II},
Selecta Math. {\bf 4} (1998), no. 2, 213--231.

\bibitem[Fe]{Fe}
Felder, M.,
{\it On the irrationality of certain coefficients of the Alekseev-Torossian associator}, J. Lie Theory 27 (2017), no. 2, 501--528. 

\bibitem[Fu1]{F} 
Furusho, H.,
{\it The multiple zeta value algebra and the stable derivation algebra},
Publ. Res. Inst. Math. Sci. Vol {\bf 39}. no 4. (2003). 695--720.

\bibitem[Fu2]{F10}
Furusho, H.,
{\it Pentagon and hexagon equations},
Ann. of Math. (2) {\bf 171} (2010), no. 1, 545--556. 
%
%

\bibitem[HLTV]{HLTV}
Hardt, R., Lambrechts, P., Turchin, V. and Voli\'{c}, I.,
{\it  Real homotopy theory of semi-algebraic sets},
Algebr. Geom. Topol. {\bf 11} (2011), no. 5, 2477--2545.

\bibitem[KaT]{KT}
Kassel, C. and Turaev, V.;
{\it Chord diagram invariants of tangles and graphs},
Duke Math. J. {\bf 92} (1998), no. 3, 497--552.

\bibitem[K]{K03}
Kontsevich, M.,
{\it Deformation quantization of Poisson manifolds}, 
Lett. Math. Phys. {\bf 66} (2003), no. 3, 157--216. 

\bibitem[KS]{KS}
Kontsevich, M. and Soibelman, Y.,
{\it Deformations of algebras over operads and the Deligne conjecture},
Conf\'{e}rence Mosh\'{e} Flato 1999, Vol. I (Dijon), 255--307, Math. Phys. Stud., {\bf 21}, Kluwer Acad. Publ., Dordrecht, 2000.

\bibitem[LV]{LV}
Lambrechts, P. and Voli\'{c}, I.,
{\it Formality of the little $N$-disks operad}, 
Mem. Amer. Math. Soc. {\bf 230} (2014), no. 1079.

\bibitem[LM1]{LM96a}
Le, T.Q.T. and Murakami, J.;
{\it The universal Vassiliev-Kontsevich invariant for framed oriented links},
Compositio Math. {\bf 102} (1996), no. 1, 41-64.

\bibitem[LM2]{LM}
Le, T.Q.T. and Murakami, J.,
{\it Kontsevich's integral for the Kauffman polynomial},
Nagoya Math. J. {\bf 142} (1996), 39--65.

\bibitem[P]{P}
Piunikhin, S.;
{\it Combinatorial expression for universal Vassiliev link invariant},
Comm. Math. Phys. {\bf 168} (1995), no. 1, 1-22. 

\bibitem[RW]{RW}
Rossi, C. and Willwacher, T.,
{\it P. Etingof's conjecture about Drinfeld associators},
preprint {\tt arXiv:1404.2047}.

\bibitem[SW]{SW} 
\v{S}evera, P. and Willwacher, T.,  
 {\it Equivalence of formalities of the little discs operad},  
 Duke Math. J. {\bf 160} (2011), no. 1, 175--206.
 
\bibitem[Ta]{Ta}
Tamarkin, D. E.;
{\it Formality of chain operad of little discs}, 
Lett. Math. Phys. {\bf 66} (2003), no. 1-2, 65--72. 

\bibitem[To]{T} 
Torossian, C.,
{\it Sur la conjecture combinatoire de Kashiwara-Vergne},
J. Lie Theory {\bf 12} (2002), no. 2, 597--616. 

\bibitem[W]{W} 
Willwacher,  T.,
 {\it M. Kontsevich's graph complex and the Grothendieck-Teichm\"{u}ller Lie algebra}, 
  Invent. Math. {\bf 200} (2015), no. 3, 671--760.

\end{thebibliography}
\end{document}